\documentclass[runningheads]{llncs}
\usepackage{graphicx}

\usepackage{amsmath,amsfonts,amssymb,epsfig,epstopdf,titling,url,array}

\numberwithin{thm}{section}

\numberwithin{alem}{section}

\begin{document}
\nocite{*}
\title{Vertex-Domatic, Edge-Domatic  and Total Domatic Number of Uniform Hypergraphs}
%
%
\author{S. P. Dash\\ Email: tin.dsm@gmail.com}
\authorrunning{S. P. Dash}
\date{}

%

\maketitle              

\begin{abstract}
E. J. Cockayne and S. T. Hedetniemi introduced the concept of domatic number of a graph. B. Zelinka extended the concept to the uniform hypergraphs. Further, B. Zelinka defined the concept of edge-domatic number and total edge-domatic number of a graph. In this paper, we investigate and prove some assertions in connection with vertex domatic number, edge-domatic number and total domatic number of some specific uniform hypergraphs. \\
\end{abstract}
		
\begin{keywords}
Hypergraphs, Domatic Number, Total domatic number,  Edge-domatic number, Total edge-domatic number
\end{keywords}
	

\section{Introduction}
The concept of domatic number of a graph was introduced in \cite{Cockayne3} and B. Zelinka \cite{Zelinka2} 
introduced the analogous definition for hypergraphs. A vertex-dominating set $V_d$ of a hypergraph $H(V,E)$ is a 
subset of $V$ such that $\forall x\in (V-V_d)$, $\exists y\in V_d$ with $x$ and $y$ are adjacent to each other. 
A partition $P$ of $V$ is a vertex-domatic partition of $H$ if each class in $P$ is a vertex-dominating set of $H$.
The maximum possible number of classes in a vertex-domatic partition is called the vertex-domatic number\cite{Cockayne3} 
of $H$ and is denoted by $d(H)$.\\

A total vertex-dominating set $V_t$, of a hypergraph $H(V,E)$ is a subset of $V$ such that $\forall x\in V$, $\exists y\in V_t$, with 
$x$ and $y$ are adjacent to each other. The total vertex-domatic number ($d_t(H)$) can be defined accordingly.\\

Throughout this paper, we denote vertex-dominating set, vertex-domatic partition and  vertex-domatic number as simply dominating set, 
domatic partition and domatic number respectively. We also denote total vertex-dominating set, 
total vertex-domatic partition and total vertex-domatic number as total dominating set, 
total domatic partition and total domatic number respectively.\\

B. Zelinka in \cite{Zelinka1} introduced the edge-domatic number of a graph. 
In this paper, we present the analogous definition for hypergraphs.
We know, two hyperedges are adjacent, if they share at least one common vertex.
An edge-dominating set $(E_d)$ of a hypergraph$H(V,E)$ is a subset of $E$, such that $\forall x_e \in (E-E_d)$, 
$\exists y_e\in E_d$ with $x_e$ and $y_e$ are adjacent to each other.
A partition $P_e$ of the edge set $E$ is an edge-domatic partition of $H$ if each class in $P_e$ is an edge-dominating set of $H$.
The maximum possible number of classes in an edge-domatic partition of $H$ is called the edge-domatic number 
and is denoted by $ed(H)$.

A total edge-dominating set $E_t$ of a hypergraph $H$ is a subset of $E$ such that $\forall x\in E$, $\exists y\in E_t$, with 
$x$ and $y$ are adjacent to each other. The total edge-domatic number ($ed_t(H)$) can be defined accordingly.\\

In this paper, we inherit the properties of hypergraphs and follow the previous definitions to present some assertions on the domatic number 
and total domatic number. Then we present some tight and lower bounds for edge-domatic number and the total edge-domatic number of some specific hypergraphs.

We have used following notations throughout this paper to make it explicit :\\
Let, $H(V,E)$ be a hypergraph.\\
$d(H)$: Domatic number of $H$.\\
$d_t(H)$: Total domatic number of $H$.\\
$ed(H)$: Edge-domatic number of $H$.\\
$ed_t(H)$: Total edge-domatic number of $H$.\\

In section-$2$, we determine domatic number and total domatic number of complete uniform hypergraphs and 
complete bipartite uniform hypergraphs. In section-$3$, we derive a lower bound on edge-domatic number for 
complete and bipartite uniform hypergraphs. Furthermore, for certain hypergraphs, we also derive the tight 
bounds of edge-domatic number.
In section-$4$, we present a lower bound on total edge-domatic number for complete uniform and complete 
bipartite uniform hypergraphs. To the best of our knowledge, these are the first non-trivial bounds on 
domatic number for uniform hypergraphs.

\section{Domatic Number and Total Domatic Number of Uniform Hypergraphs}
The following two theorems concern the domatic number for complete uniform hypergraphs and complete bipartite uniform hypergraphs.
\begin{theorem}
Let $H(V,E)$ be a complete $r$-uniform hypergraph with $n$ vertices, then\\$d(H)=n$ and $d_t(H)=\lfloor \frac{n}{2}\rfloor$.
\end{theorem}

\begin{proof}
$H$ is complete, so each vertex is adjacent to all other vertices. 
Evidently, each vertex would constitute a dominating set $V_d$ of $H$ and 
we can partition the vertex set $V$ into $n$ disjoint dominating sets.  
This is the maximum possible domatic partition of $V$. Hence, $d(H)=n$. 

For the total domatic partition, consider two arbitrary vertices $x$ and $y$. 
The vertices $x$ and $y$ are adjacent to each other and also adjacent to all other vertices in $H$.  
Thus, every vertex in $H$ is adjacent to at least one vertex in $\{x,y\}$. 
As a result, the set $\{x, y\}$ is a total dominating set. 
Consequently, we can partition the vertex set $V$ into $\lfloor \frac{n}{2}\rfloor$ mutually exclusive pairs 
such that each pair would be a total dominating set. 
In case $n$ is odd, after constructing $\lfloor \frac{n}{2}\rfloor$ mutually 
exclusive pairs we can assign the leftover vertex to any pair. 
Thus, $d_t(H)\geq \lfloor \frac{n}{2}\rfloor$.
Since one vertex can not be adjacent to itself, we can not construct a total dominating set 
with a single vertex. Hence, $d_t(H)= \lfloor \frac{n}{2}\rfloor$.
\end{proof}
\begin{theorem}
For any complete bipartite $r$-uniform ($3\leq r\leq n$) hypergraph with $n(n\geq 3)$ vertices, the domatic number 
and the total domatic number are $n$ and $\lfloor \frac{n}{2}\rfloor$ respectively.
\end{theorem}

\begin{proof}
Let, $H(X,Y,E)$ be a complete bipartite $r$-uniform hypergraph with $n$ vertices. 
Since the value of $r$ is at least $3$, any two vertices either from $X$ or from $Y$ would be part of some hyperedge. 
Now, consider two arbitrary vertices $x$ and $y$ such that $x\in X$ and $y\in Y$. 
Clearly, $x$ and $y$ are also part of some hyperedges in $H$. 
We can conclude that, each vertex in $H$ is adjacent to all other vertices. 
Hence, each vertex would be a dominating set, so the domatic number $d(H)=n$.
Like the previous theorem, we can prove that each vertex pair $(x,y)$ in $H$ would be a total dominating set 
and the total domatic number $d_t(H)$ is $\lfloor \frac{n}{2}\rfloor$.\\
\end{proof}

\section{Edge-Domatic Number of Uniform Hypergraphs}
In this section, we illustrate the inequalities and some tight bounds of the edge-domatic number for certain uniform hypergraphs.
\begin{lemma}
Let $H(V,E)$ be a $r$-uniform hypergraph without an isolated vertex, having $n$ vertices and $m$ hyperedges. 
If $r>\lfloor \frac{n}{2}\rfloor$ then, $ed(H)=m$.
\end{lemma}

\begin{proof}
Consider two arbitrary hyperedges $e_i$ and $e_j$. As $r>\lfloor \frac{n}{2}\rfloor$, by pigeonhole principle $e_i$ and $e_j$ have at least one common vertex. 
Therefore, every hyperedge is adjacent to all other hyperedges. Thus, each hyperedge would constitute an edge-dominating set. 
Henceforth, $ed(H)=m$. 
\end{proof}

\begin{theorem}
Let, $H(V,E)$ be a complete $r$-uniform hypergraph with $n$ vertices then,\\
$ed(H)=C(n,r)$, if $r>\lfloor \frac{n}{2}\rfloor$\\
\hspace*{.4in}$=\frac{C(n,r)}{\frac{n}{r}} $, if $r\leq \lfloor \frac{n}{2}\rfloor$ and $n$ is divisible by $r$.\\
\end{theorem}

\begin{proof}
The number of hyperedges of a complete $r$-uniform hypergraph with $n$ vertices is $C(n,r)$.\\
For $r>\lfloor \frac{n}{2}\rfloor$: from the previous lemma we know that each edge-domatic partition 
consists of one hyperedge. Hence, $ed(H)=C(n,r)$.

For $r\leq \lfloor \frac{n}{2}\rfloor$ and $n$ is divisible by $r$:\\ 
Let, $V=\{v_1, v_2, \cdots, v_n\}$. Consider a subset of hyperedges;
$S=\{e_1, e_2, \cdots, e_{n/r}\}$ such that,\\
$e_1=\{v_1,v_2, \cdots, v_r\}, e_2=\{v_{r+1},v_{r+2}, \cdots, v_{2r}\},\cdots, e_{n/r}=\{v_{n-r+1}, \cdots, v_n\}$. 
We can notice that, $S$ covers all vertices of $H$, so each hyperedge in $(E-S)$ is adjacent to at least one edge in $S$.
As a result, $S$ becomes an edge-dominating set. Similarly, we can form different mutually exclusive sets of $(\frac{n}{r})$ hyperedges, 
each of these sets would cover all vertices of $H$. Moreover, each of these sets are edge dominating sets of $H$. 
Hence, $ed(H)\geq \frac{C(n,r)}{\frac{n}{r}}$. Suppose, $ed(H)\geq \frac{C(n,r)}{\frac{n}{r}}+1$. 
Then, there must be some edge-dominating sets that 
would have at most $(\frac{n}{r} -1)$ hyperedges. Let, $C$ be such an edge-dominating set. 
Clearly, $C$ covers at most $(n-r)$ vertices. As a result, there would be at least $r$ vertices left untraversed. 
The hyperedge containing those $r$ vertices is adjacent to none of the hyperedges in $C$. 
Therefore, $C$ can not be an edge-dominating set. 
Hence, $ed(H)=\frac{C(n,r)}{\frac{n}{r}}$.\\

\end{proof}

\begin{theorem}
Let, $H(X,Y,E)$ be a complete bipartite $r$-uniform hypergraph,\\
ed(H)=$\frac{|E|}{2},  \text{if } |X|=|Y|=r$.\\
\hspace*{.3in} $\geq \frac{|E|}{2k}$, if $|X|=|Y|=kr$, $k$ is a constant and $k\geq 1$.\\
\hspace*{.3in} $\geq \frac{|E|}{2*\frac{min(|X|, |Y|)}{r}}$, if $|X|=k_1r$ and $|Y|=k_2r$, $k_1\neq k_2$.\\
\end{theorem}
\begin{proof}
\textbf{Case 1:}
Here $|X|=|Y|=r$: Let $X=\{x_1, x_2, \cdots, x_r\}$ and $Y=\{y_1, y_2, \cdots, y_r\}$. 
Consider a set $C$ of two hyperedges, namely $e_1$ and $e_2$. Specifically, \\ $e_1=\{x_1, y_1, y_2, \cdots, y_{r-1} \}$ and 
$e_2=\{x_2, x_3,\cdots, x_r, y_r \}$. As $C$ covers all vertices of $H$, so all hyperedges in $(E-C)$ are 
adjacent to both $e_1$ and $e_2$. Thus $C$ is an edge-dominating set. Consequently, we can partition $E(H)$ into $(|E|/2)$ 
mutually exclusive edge-dominating sets such that each consists of two vertex disjoint hyperedges. 
Thus, $ed(H)\geq \frac{|E|}{2}$. 
Assume for the sake of contradiction that $ed(H)\geq \frac{|E|}{2}+1$. 
This implies there must be an edge dominating set $C_1$ with one hyperedge. As a result, 
$C_1$ covers $r$ vertices. Moreover, the hyperedge in $C_1$ is not adjacent to the hyperedge consists of rest $r$ vertices. 
Therefore, $C_1$ is not an edge dominating set. This is a contradiction. Hence, $ed(H)=\frac{|E|}{2}$.\\

\textbf{Case 2:} $|X|=|Y|=kr$, $k$ is a constant and $k\geq 1$. For simplicity purpose assume $k=2$. 
Let $X={x_1, x_2,\cdots, x_{2r}}$ and $Y={y_1, y_2,\cdots, y_{2r}}$. 
Now, we construct an edge dominating set $C=\{e_1, e_2, e_3, e_4\}$ such that,\\
$e_1=\{x_1, y_1, y_2, \cdots, y_{r-1} \}, e_2=\{x_2, x_3,\cdots, x_r, y_r \}, 
e_3=\{x_{r+1}, y_{r+1}, y_{r+2}, \cdots, y_{2r-1} \}$ and $e_4=\{x_{r+2}, x_{r+3},\cdots, x_{2r}, y_{2r} \}$. 
Note that $C$ covers all vertices of $H$, so all hyperedges in $(E-C)$ are adjacent to at least one hyperedge in $C$. 
Hence $C$ is an edge-dominating set.
Likewise, we can partition $E$ into $|E|/4$ edge-dominating sets such that, each consists of four vertex-disjoint hyperedges. 
In general, for an arbitrary value of $k$ we can construct edge-dominating sets with $2k$ hyperedges so that, 
each edge-dominating set covers all vertices of $H$. Subsequently, $E(H)$ can be partitioned into $|E|/2k$ edge-dominating sets. 
Hence, ed$(H)\geq \frac{|E|}{2k}$.\\

\textbf{Case 3:} Here, $|X|=k_1r$ and $|Y|=k_2r$, $k_1\neq k_2$. Without loss of generality, 
let us assume $k_1<k_2$. 
Now, we construct an edge-dominating set $C$ by taking $k_1=2$ and $k_2=3$ as follows: 
Let, $X={x_1, x_2,\cdots, x_{2r}}$ and $Y={y_1, y_2,\cdots, y_{3r}}$.\\
$C=\{e_1, e_2, e_3, e_4\}$ such that,\\ $e_1=\{x_1, y_1, y_2, \cdots, y_{r-1} \}, e_2=\{x_2, x_3,\cdots, x_r, y_r \}, 
e_3=\{x_{r+1}, y_1, y_2, \cdots, y_{r-1} \}$ and $e_4=\{x_{r+2}, x_{r+3},\cdots, x_{2r}, y_r \}$. 
It is clear that $C$ covers all vertices in $X$. Thus, every hyperedge in $(E-C)$ are adjacent to at 
least one hyperedge in $C$. This implies $C$ is an edge-dominating set. Likewise, we can construct 
other edge dominating sets by taking four hyperedges that cover all vertices in $X$. 
For arbitrary values of $k_1$ and $k_2$, edge dominating sets can be constructed 
by taking $2k_1$ hyperedges those covers all vertices in $X$. 
Hence, $ed(H)\geq \frac{|E|}{2k_1}$.\\
\end{proof}

\paragraph{}Following lemma presents a better bound for the edge-domatic number of 
a complete bipartite $3$-uniform hypergraph.\\




\begin{lemma}
The edge-domatic number of a complete bipartite $3$-uniform hypergraph $H(X,Y,E)$ with even number of vertices in each subset and $|X|=|Y|=k$ is 
$\frac{|E|}{k/2}$.
\end{lemma}

\begin{proof}
Consider a subset $C_1$ of hyperedges of $H$ having $\frac{k}{2}$ hyperedges such that $(i)$ each hyperedge in $C_1$ takes two vertices 
from $X$ and one vertex from $Y$. $(ii)$ they have no common vertex in $X$. 
$(iii)$ they share a common vertex in $Y$. 
Note that $C_1$ covers all vertices in $X$. This implies each hyperedge in $(E-C_1)$ is adjacent to at least one hyperedge in $C_1$. 
Thus, $C_1$ is an edge-dominating set. 

Similarly, we can generate another edge-dominating set $C_2$ with 
$\frac{k}{2}$ hyperedges suh that each hyperedge has two vertices from $Y$ and one vertex from $X$. Furthermore, 
$C_2$ covers all vertices in $Y$. 

Now, we can partition $E$ into different edge-dominating sets, each having 
$\frac{k}{2}$ hyperedges. Consequently, $ed(H)\geq \frac{|E|}{k/2}$. 

Assume for the sake of contradiction that $ed(H)\geq \frac{|E|}{k/2}+1$. 
This implies there exist some edge-dominating sets with less than $\frac{k}{2}$ hyperedges. Let, $C_3$ be such an edge-dominating set. 
Clearly, $C_3$ can cover at most $(k-2)$ vertices from both $X$ and $Y$ and leaving at least two vertices untraversed 
from both $X$ and $Y$. Since, we would have a hyperedge consisting of those untraversed vertices, $C_3$ can not 
be an edge-dominating set, leading to a contradiction. Hence, $ed(H)=\frac{|E|}{k/2}$.
\end{proof}

\begin{lemma}
Let $H(X,Y,E)$ be a complete bipartite $r$-uniform hypergraph. If $|X|=|Y|=r$ and $\delta$ be the 
degree of each vertex. Then, ed($H$)=$\delta$.
\end{lemma}

\begin{proof}
The hypergraph $H$ has $2r$ vertices each having degree $\delta$, so the sum of degrees 
is $2\delta r$. As, each hyperedge contributes a total of $r$ degrees to $r$ different vertices, 
$|E|=\frac{2\delta r}{r}=2\delta$. 
From case-$1$ of theorem-$5$, $ed(H)=\frac{|E|}{2}=\frac{2\delta}{2}=\delta$.
\end{proof}

\section{\textbf{Total Edge-Domatic Number of Uniform Hypergraphs}}
The results in this section concern the total edge domatic number for complete and complete 
bipartite uniform hypergraphs.
\begin{lemma}
Let $H(V,E)$ be a $r$-uniform hypergraph without an isolated vertex, having $n$ vertices and $m$ hyperedges. 
If $r>\lfloor \frac{n}{2}\rfloor$ then, $ed_t(H)=\lfloor \frac{m}{2}\rfloor$.
\end{lemma}

\begin{proof}
Consider, an arbitrary subset $C$ of two hyperedges, $C=\{e_i$, $e_j\}$. 
As $r>\lfloor \frac{n}{2}\rfloor$, by pigeonhole principle $e_i$ and $e_j$ have at least one common vertex. Apparently, every hyperedge 
of $H$ is adjacent to at least one hyperedge in $C$. Thus, $C$ becomes a total edge-dominating set. 
Now, we can partition the edge set into $\lfloor \frac{m}{2}\rfloor$ number of classes 
such that each class would be a total edge-dominating set. In fact, each class having two hyperedges if $m$ is even and 
in case $m$ is odd, the extra hyperedge can be allocated to any of the classes. 
Hence, $ed_t(H)=\lfloor \frac{m}{2}\rfloor$.
\end{proof}

\begin{theorem}
Let, $H(V,E)$ be a complete $r$-uniform hypergraph with $n$ vertices. 
If, $r>\lfloor \frac{n}{2}\rfloor$ then, $ed_t(H)=\lfloor \frac{C(n,r)}{2}\rfloor$.\\
\end{theorem}

\begin{proof}
As, $H$ is complete, it has no isolated vertex. Now, applying previous lemma, 
$ed_t(H)=\lfloor\frac{C(n,r)}{2}\rfloor$.
\end{proof}

\begin{theorem}
Let, $H(X,Y,E)$ be a complete bipartite $r$-uniform hypergraph with $r$ vertices in each subset. 
Then $ed_t(H)=\frac{|E|}{2}$.\\
\end{theorem}

\begin{proof}
Consider two arbitrary hyperedges $e_i$ and $e_j$ such that they have at least one common vertex. 
Thus, they are adjacent to each other.  It is clear that $e_i$ and $e_j$ covers at least $(r+1)$ vertices  
and they do not cover at most $(r-1)$ vertices. Therefore, each hyperedge has at least one common vertex with 
$e_i$ or $e_j$ or both. This implies each hyperedge in $H$ is adjacent to either $e_i$ 
or $e_j$ or both. Moreover, the set $\{e_i,e_j\}$ constitutes a total edge dominating set. 
Consequently, we can partition $E$ into $|E|/2$ number of total edge-dominating sets such that 
each contains two hyperedges with at least one common vertex. 
We can not a have total edge-dominating set with only one hyperedge because, one vertex can not be adjacent to 
itself. Hence, $ed_t(H)=\frac{|E|}{2}$ (we can notice that |E| is even here).\\
\end{proof}



\begin{lemma}
The total edge-domatic number of a complete bipartite $3$-uniform hypergraph $H(X,Y,E)$ with even number of vertices in each subset and $|X|=|Y|=k$ is 
$\frac{|E|}{k/2}$.
\end{lemma}

\begin{proof}
Here we use the edge-dominating set constructed in lemma-$6$. It is sufficient to prove, 
the edge-dominating set $C$ in lemma-$6$ is a total edge-dominating set. 
We know each hyperedge in $(E-C)$ is adjacent to at least one hyperedge in $C$. 
As hyperedges in $C$ have a common vertex, so they are adjacent to each other. 
Consequently, each hyperedge in $H$ is adjacent to at least one hyperedge in $C$. 
Thus, $C$ becomes a total edge-dominating set and $ed_t(H)\geq \frac{|E|}{k/2}$. 
It follows from lemma-$6$ that we can not have a total edge-dominating set with 
less than $k/2$ hyperedges. Hence, $ed_t(H)=\frac{|E|}{k/2}$.
\end{proof}

\begin{lemma}
Let $H(X,Y,E)$ be a complete bipartite $r$-uniform hypergraph. If $|X|=|Y|=r$ and $\delta$ be the 
degree of each vertex of $H$. Then $ed_t(H)=\delta$.\\
\end{lemma}

\begin{proof}
It can be noted from theorem-$10$ that $ed_t(H)=\frac{|E|}{2}$ and from lemma-$8$, $|E|=2\delta$. 
Hence, $ed_t(H$) becomes $\delta$.
\end{proof}

\section{Conclusion}
Although we found the non-trivial bounds on domatic and edge-domatic numbers for some specific uniform 
hypergraphs, the problem is still open for other classes of hypergraphs.
\bibliographystyle{elsarticle-num.bst}

\end{document}